\documentclass[12pt]{article}
\usepackage[fleqn]{amsmath}
\usepackage{amssymb,latexsym,color,dsfont}
\usepackage[mathscr]{eucal}
\usepackage[margin=1.1in]{geometry}
\usepackage[auth-sc]{authblk}
\usepackage{tikz}
\usepackage{pgf}
\usetikzlibrary{arrows}
\usepackage[utf8]{inputenc}
\usepackage[T2A]{fontenc}
\usepackage[english,russian]{babel}

\newtheorem{theorem}{Theorem}[section]
\newtheorem{lemma}[theorem]{Lemma}
\newtheorem{proposition}[theorem]{Proposition}

\newenvironment{proof}{\noindent
  \textbf{Proof.}}{\hfill$\Box$\\}

\newcommand{\vp}{\ensuremath{\varphi}}

\newcommand{\con}{\wedge}

\newcommand{\dis}{\vee}

\newcommand{\imp}{\rightarrow}

\newcommand{\bottom}{\perp}

\newcommand{\mmodel}[1]{\ensuremath{\frak{#1}}}

\newcommand{\sat}[3]{\ensuremath{\frak{#1}, #2 \models #3}}

\newcommand{\noop}[1]{}

\begin{document}

\title{Recursive enumerability and elementary frame definability in
  predicate modal logic\thanks{A pre-final draft of a paper to appear
    in \textit{Journal of Logic and Computation}, DOI
    \mbox{10.1093/logcom/exz028}.}} \author{Mikhail
  Rybakov\thanks{Institute for Information Transmission Problems,
    Russian Academy of Sciences, and Higher School of Economics,
    Moscow, Russia, \texttt{m\_rybakov@mail.ru}} and Dmitry
  Shkatov\thanks{School of Computer Science and Applied Mathematics,
    University of the Witwatersrand, Johannesburg, South Africa,
    \texttt{shkatov@gmail.com}}}

\selectlanguage{english}
\date{\today}
\maketitle

\begin{abstract}
  We investigate the relationship between recursive enumerability and
  elementary frame definability in first-order predicate modal logic.
  On the one hand, it is well-known that every first-order predicate
  modal logic complete with respect to an elementary class of Kripke
  frames, i.e., a class of frames definable by a classical first-order
  formula, is recursively enumerable.  On the other, numerous examples
  are known of predicate modal logics, based on ``natural''
  propositional modal logics with essentially second-order Kripke
  semantics, that are either not recursively enumerable or Kripke
  incomplete.  This raises the question of whether every Kripke
  complete, recursively enumerable predicate modal logic can be
  characterized by an elementary class of Kripke frames.  We answer
  this question in the negative, by constructing a normal predicate
  modal logic which is Kripke complete, recursively enumerable, but
  not complete with respect to an elementary class of frames.  We also
  present an example of a normal predicate modal logic that is
  recursively enumerable, Kripke complete, and not complete with
  respect to an elementary class of rooted frames, but is complete
  with respect to an elementary class of frames that are not rooted.
\end{abstract}

\section{Introduction}
\label{sec:into}

It has been observed (see, e.g., \cite{RShAiML}) that first-order
predicate modal logics built on top of propositional modal logics with
essentially second-order Kripke semantics usually exhibit some
undesirable properties.  Indeed, for all ``natural'' propositional
modal logics with essentially second-order Kripke semantics known from
the literature---such as {\bf GL} (G\"{o}del-L\"{o}b), {\bf Grz}
(Grzegorczyk), as well as their ``linear'' counterparts {\bf GL.3} and
{\bf Grz.3}; the logic of finite Kripke frames; propositional dynamic
logics; epistemic logics with the common knowledge operator; and
branching-time temporal logics {\bf CTL} and {\bf CTL$^\ast$}---the
sets of predicate modal formulas valid on their frames are not
recursively enumerable~\cite{Rybakov01,Skv88,Skvortsov95,WZ01}, while
the logics obtained by adding to their representations as
Hilbert-style calculi, in cases where there exists one, the axioms and
inference rules of the classical first-order logic are Kripke
incomplete~\cite{Montagna84,Rybakov01,Skv88,Skvortsov95,WZ01}.  Thus,
it would appear that, in predicate modal logic, completeness with
respect to Kripke semantics with essentially second-order conditions
and recursive enumerability might be incompatible.  This, in
particular, raises the question of whether for Kripke complete
predicate modal logics recursive enumerability and completeness with
respect to an elementary class of frames coincide.\footnote{A similar
  question can be posed for $\Delta$-elementary classes of frames,
  i.e., classes of frames defined by sets of first-order formulas;
  this is not the notion we consider in the present paper.}

A partial answer was given in~\cite{RShAiML}, where it was shown that
recursive enumerability and completeness with respect to an elementary
class of frames do not coincide for Kripke complete quasi-normal
predicate modal logics; furthermore, it was assumed in~\cite{RShAiML},
without an explicit mention, that logics under consideration were
exclusively those determined by rooted frames, i.e., frames generated
by a single world.  The question of whether recursive enumerability
implies completeness with respect to an elementary class of frames for
Kripke complete normal predicate modal logics has, however, remained
open.

In the present paper, we answer that question in the negative, by
exhibiting an example of a normal predicate modal logic that is Kripke
complete, recursively enumerable, but not complete with respect to an
elementary class of Kripke frames.  We also show that the assumption
of all frames for a logic being rooted
may play a crucial role when answering this question: we construct a
normal predicate logic that enjoys the required properties if we
meta-logically restrict our attention to rooted frames, but ceases to
enjoy them once all frames are taken into consideration.  This might
be of independent interest, as for most purposes in modal logic, it
suffices to only consider rooted frames.

The paper is structured as follows. In Section~\ref{sec:prelim}, we
introduce the necessary preliminaries.  In
Section~\ref{sec:example-1}, we present an example of a normal
predicate modal logic that is Kripke complete, recursively enumerable,
but not complete with respect to an elementary class of rooted frames;
we also show that this logic \textit{is} complete with respect to an
elementary class containing frames that are not required to be rooted.
Then, in Section~\ref{sec:example-2}, we present an example of a
normal predicate modal logic that is Kripke complete, recursively
enumerable, but not complete with respect to an elementary class of
\textit{any} frames.  We conclude in Section~\ref{sec:discussion}.

\section{Preliminaries}
\label{sec:prelim}

A first-order predicate modal language contains countably many
individual variables; countably many predicate letters of every arity,
including zero; Boolean connectives $\neg$ and $\wedge$; the unary
modal connective $\Box$; and the quantifier $\forall$.  Atomic
formulas, formulas, as well as the symbols $\vee$, $\imp$, $\exists$,
and $\Diamond$ are defined in the usual way.  We assume that $\vee$
and $\wedge$ bind stronger than $\imp$.

Given a formula $\varphi$, we denote by $md(\varphi)$ the modal depth
of $\varphi$, which is defined inductively, as follows:
$$
\begin{array}{rcl}
  md(\alpha) & = & 0, \mbox{ where } \alpha \mbox{ is an atomic formula}; \\
  md(\neg \vp_1) & = & md(\vp_1); \\
  md(\vp_1 \con \vp_2) & = & \max\{md(\vp_1), md(\vp_2)\}; \\
  md(\forall y\,\vp_1) & = & md(\vp_1); \\
  md(\Box \vp_1) & = & md(\vp_1)\,{+}\,1.
\end{array}
$$

We also inductively define, for every $n \in \mathds{N}$ and every
formula $\vp$, the formulas $\Box^n\vp$ and $\Diamond^n \vp$, as
follows: $\Box^0 \vp = \vp$, $\Box^{n+1} \vp = \Box \Box^{n} \vp$;
$\Diamond^0 \vp = \vp$,
$\Diamond^{n+1} \vp = \Diamond \Diamond^{n} \vp$.

A {\em normal predicate modal logic} is a set $L$ of formulas
containing the validities of the classical first-order logic as well
as the formulas of the form
$\Box (\vp \imp \psi ) \imp ( \Box \vp \imp \Box \psi )$, and closed
under predicate substitution, modus ponens, generalization (if
$\varphi \in L$, then $\forall y\, \varphi \in L$), and necessitation
(if $\varphi \in L$, then $\Box \varphi \in L$).

Modal formulas can be interpreted using Kripke semantics. A {\em
  Kripke frame} is a tuple $\frak{F} = \langle W,R\rangle$, where $W$
is a non-empty set of {\em worlds} and $R$ is a {\em binary
  accessibility relation} on $W$.  A {\em predicate Kripke frame} is a
tuple $\frak{F}_D = \langle W,R, D\rangle$, where $\langle W,R\rangle$
is a frame and $D$ is a function from $W$ into a set of non-empty
subsets of some set, the {\em domain} of~$\frak{F}_D$, satisfying the
condition that $wRw'$ implies $D(w) \subseteq D(w')$. We call the set
$D(w)$ the {\em domain} of $w$. If a predicate frame satisfies the
condition that $wRw'$ implies $D(w) = D(w')$, we refer to it as a {\em
  predicate frame with constant domains}.

Note that a Kripke frame can be considered, in an obvious way, as a
model for the classical first-order language in the signature
$\{R, = \}$.

A frame $\frak{F} = \langle W,R\rangle$ is {\em rooted} if there
exists $w_0 \in W$ such that $w_0 R^\ast w$ holds for every $w \in W$,
where $R^\ast$ is the reflexive and transitive closure of $R$. Given a
frame $\langle W,R\rangle$ and $w, w' \in W$, we say that $w'$ is {\em
  accessible from $w$} or that $w$ {\em sees} $w'$ if $w R w'$ holds.
We say that $w \in W$ is a {\em dead end} if $w R v$ does not hold for
any $v \in W$.  If $\frak{C}$ is a class of frames and
$\frak{F} \in \frak{C}$, we refer to $\frak{F}$ as a {\em
  $\frak{C}$-frame}.

A {\em Kripke model} is a tuple $\frak{M} = \langle W,R,D,I\rangle$,
where $\langle W,R, D\rangle$ is a predicate Kripke frame and $I$ is a
function assigning to a world $w\in W$ and an $n$-ary predicate letter
$P$ an $n$-ary relation $I(w,P)$ on $D(w)$.  We refer to $I$ as an
{\em interpretation of predicate letters} with respect to worlds in
$W$.

We say that a model $\langle W,R,D,I\rangle$ is {\em based on} the
predicate frame $\langle W,R,D\rangle$ and the frame
$\langle W,R\rangle$; similarly, we say that a predicate frame
$\langle W,R,D\rangle$ is {\em based on} the frame
$\langle W,R\rangle$.  We also say that $\langle W,R,D\rangle$ is the
{\em underlying predicate frame} for a model $\langle W,R,D,I\rangle$.

An {\em assignment} in a model is a function $g$ associating with
every individual variable $y$ an element $g(y)$ of the domain of the
underlying predicate frame.

The truth of a formula $\varphi$ at a world $w$ of a model $\frak{M}$
under an assignment $g$ is defined inductively, as follows:
\begin{itemize}
\item
$\frak{M},w\models^g P(y_1,\ldots,y_n)$ if $\langle
g(y_1),\ldots,g(y_n)\rangle \in I(w,P)$;
\item $\frak{M},w\models^g\neg\varphi_1$ if
  $\frak{M},w\not\models^g\varphi_1$;
\item $\frak{M},w\models^g\varphi_1\wedge\varphi_2$ if
  $\frak{M},w\models^g\varphi_1$ and $\frak{M},w\models^g\varphi_2$;

\item $\frak{M},w\models^g\Box\vp_1$ if $wRw'$ implies
  $\frak{M},w'\models^g\varphi_1$, for every $w'\in W$;
\item $\frak{M},w\models^g\forall y\,\varphi_1$ if
  $\frak{M},w\models^{g'}\varphi_1$, for every assignment $g'$ such
  that $g'$ differs from $g$ in at most the value of $y$ and such that
  $g'(y)\in D(w)$.
\end{itemize}




We say that a formula $\vp$ is {\em true at a world $w$} of a model
$\frak{M}$ and write $\frak{M},w\models \vp$ if
$\frak{M},w\models^g \vp$ holds for every $g$ assigning to free
variables of $\vp$ elements of $D(w)$.  We say that $\vp$ is {\em
  valid at a world $w$} of a frame $\frak{F}$ if
$\frak{M},w\models \vp$ holds for every model $\frak{M}$ based on
$\frak{F}$.  We say that $\vp$ is {\em true in $\frak{M}$} and write
$\frak{M} \models \vp$ if $\frak{M},w\models \vp$ holds for every
world $w$ of $\frak{M}$.  We say that $\vp$ is {\em valid on a
  predicate frame $\frak{F}_D$} and write $\frak{F}_D \models \vp$ if
$\vp$ is true in every model based on $\frak{F}_D$. We say that $\vp$
is {\em valid on a frame $\frak{F}$} and write $\frak{F} \models \vp$
if $\vp$ is valid on every predicate frame
$\langle \frak{F}, D \rangle$.  We say that a set of formulas $\Gamma$
is {\em valid on a frame $\frak{F}$} and write
$\frak{F} \models \Gamma$ if $\frak{F} \models \vp$ holds for every
$\vp \in \Gamma$.  Finally, we say that a set of formulas is {\em
  valid on a class of frames} if it is valid on every frame from the
class.

Let $\frak{M} = \langle W,R,D,I\rangle$ be a model, $w \in W$, and
$a_1, \ldots, a_n \in D(w)$.  Let $\vp(y_1, \ldots, y_n)$ be a formula
whose free variables are among $y_1, \ldots, y_n$.  We write
$\frak{M}, w \models \vp [a_1, \ldots, a_n]$ to mean that
$\frak{M}, w \models^g \vp (y_1, \ldots, y_n)$, where
$g(y_1) = a_1, \ldots, g(y_n) = a_n$.

Given a class of frames $\frak{C}$, the set of predicate modal
formulas valid on every $\frak{C}$-frame is denoted by $L(\frak{C})$;
this set is a normal predicate modal logic.

A normal predicate modal logic $L$ is {\em sound and complete} with
respect to a class of frames $\frak{C}$ if $L = L(\frak{C})$; in this
case, for brevity, we also say that $L$ is complete with respect to
$\frak{C}$.  A logic is {\em Kripke complete} if it is sound and complete
with respect to some class of frames.

A class $\frak{C}$ of frames is {\em elementary} if there exists a
closed classical first-order formula $\Phi$ in the signature
$\{R, = \}$
such that $\mathfrak{F} \in \frak{C}$
if, and only if, $\Phi$ is true in $\mathfrak{F}$ considered as a
classical model, in which case we write $\mathfrak{F} \Vdash \Phi$.

The following proposition is well-known (see~\cite{ChR00},
\cite[Proposition 3.12.8]{GShS}).

\begin{proposition}
  \label{prop:re}
  Let $\frak{C}$ be an elementary class of frames.  Then,
  $L(\frak{C})$ is recursively enumerable.
\end{proposition}

In the strict sense, the converse of Proposition~\ref{prop:re} is not
true, as some recursively enumerable logics are not complete with
respect to any class of frames whatsoever (i.e., are Kripke
incomplete), and thus, are not complete with respect to any elementary
class; some examples have been mentioned in the Introduction (Section
1).  A more interesting question is whether every recursively
enumerable Kripke complete logic is a logic of an elementary class of
frames.

The main contribution of this paper is in showing, which we do in
Section~\ref{sec:example-2}, that this is not so for normal predicate
logics---namely, we exhibit an example of a recursively enumerable
Kripke complete normal predicate logic that is not complete with
respect to an elementary class of frames.  Prior to that, however, we
consider, in the next section, a similar question for rooted frames.

\section{An example over rooted frames}
\label{sec:example-1}

In this section, we exhibit the normal predicate modal logic $L_0$
which is recursively enumerable, Kripke complete, but not complete
with respect to an elementary class of \textit{rooted} frames.  The
restriction to rooted frames is sufficient for most purposes in modal
logic; we show, however, later on in this section that in the context
of the present inquiry this restriction does matter---while $L_0$ is
not complete with respect to an elementary class of rooted frames, it
is complete with respect to an elementary class of frames that are not
required to be rooted.  In the next section, we discard the
restriction to rooted frames; the ideas introduced in this section
are, however, reused in that context.

Let $\frak{F}_n$, for $n \geqslant 1$, be the frame
$\langle W_n, R_n \rangle$, where
$W_n = \{ w_1, \ldots, w_n, w^\ast \}$ and
$R_n = \{ \langle w_i, w_{i+1} \rangle : 1 \leqslant i < n \} \cup \{
\langle w_1, w^\ast \rangle \}$;
the frame $\frak{F}_n$ is depicted in \mbox{Figure}~\ref{fig}.  Denote
the set of all such frames by $\frak{C}^\ast$, and let
$\frak{C}_0 = \{ \frak{F}_{2n} \in \frak{C}^\ast : n \geqslant 1 \}$;
finally, let $L_0 = L(\frak{C}_0)$.

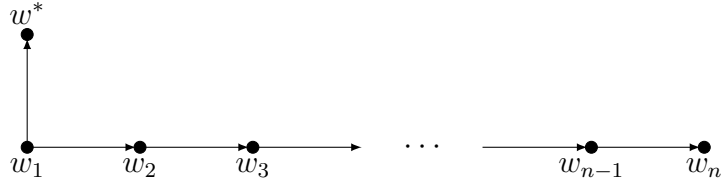
\begin{figure}
\centering
\begin{tikzpicture}[scale=1.5]

\coordinate (w1)   at (+0.0000,+0.0000);
\coordinate (w2)   at (+1.0000,+0.0000);
\coordinate (w3)   at (+2.0000,+0.0000);
\coordinate (w4ph) at (+3.0000,+0.0000);
\coordinate (dts)  at (+3.5000,+0.0000);
\coordinate (w5ph) at (+4.0000,+0.0000);
\coordinate (wn-1) at (+5.0000,+0.0000);
\coordinate (wn)   at (+6.0000,+0.0000);
\coordinate (w*)   at (+0.0000,+1.0000);

\begin{scope}[>=latex]
\draw [->,  shorten >= 1.5pt, shorten <= 1.5pt]
(w1) -- (w2);
\draw [->,  shorten >= 1.5pt, shorten <= 1.5pt]
(w2) -- (w3);
\draw [->,  shorten >= 1.5pt, shorten <= 1.5pt]
(w3) -- (w4ph);
\draw [->,  shorten >= 1.5pt, shorten <= 1.5pt]
(w5ph) -- (wn-1);
\draw [->,  shorten >= 1.5pt, shorten <= 1.5pt]
(wn-1) -- (wn);
\draw [->,  shorten >= 1.5pt, shorten <= 1.5pt]
(w1) -- (w*);
\end{scope}

\node [below      ] at (w1)   {${w_1}$}     ;
\node [below      ] at (w2)   {${w_2}$}     ;
\node [below      ] at (w3)   {${w_3}$}     ;
\node [           ] at (dts)  {${\cdots}$}  ;
\node [above      ] at (w*)   {${w^\ast}$}  ;
\node [below      ] at (wn-1) {${w_{n-1}}$} ;
\node [below      ] at (wn)   {${w_n}$}     ;


\filldraw [] (w1)   circle [radius=1.5pt]   ;
\filldraw [] (w2)   circle [radius=1.5pt]   ;
\filldraw [] (w3)   circle [radius=1.5pt]   ;
\filldraw [] (w*)   circle [radius=1.5pt]   ;
\filldraw [] (wn-1) circle [radius=1.5pt]   ;
\filldraw [] (wn)   circle [radius=1.5pt]   ;

\end{tikzpicture}

\caption{Frame $\frak{F}_n$}
\label{fig}
\end{figure}

The logic $L_0$ is Kripke complete by definition.  We next show that
$L_0$ is recursively enumerable and that $L_0$ is not complete with
respect to an elementary class of rooted frames.

To show that $L_0$ is recursively enumerable, we effectively embed it
into the classical first-order logic with equality ${\bf QCl}_=$,
whose set of theorems is known to be recursively enumerable.

Let $R$ and $D$ be binary, and $W$ unary, predicate letters not
occurring in $\varphi$; intuitively, $W(x)$ means that $x$ is a world,
$D(x,y)$ means that $y$ is an element of the domain of $x$, and
$R(x,y)$ means that $y$ is accessible from $x$.

Let $ST_x(\varphi)$ be the standard translation (see~\cite{ChR00},
\cite[Section 3.12]{GShS}) of the predicate modal formula $\varphi$
into the language of ${\bf QCl}_=$, defined as follows:
$$
\begin{array}{rcl}
\label{STx_modal}

  ST_x(P(y_1,\dots,y_m)) & = & P'(y_1,\dots,y_m,x); \\
  ST_x(\vp_1\con\vp_2) & = & ST_x(\vp_1) \con ST_x(\vp_2); \\
  ST_x(\neg \vp_1) & = & \neg ST_x(\vp_1); \\
  ST_x(\Box \vp_1) & = & \forall y\,(W(y)\wedge R(x,y)\imp
                     ST_y( \vp_1 )); \\
  ST_x(\forall y\, \vp_1) & = & \forall y\,(\neg W(y)\wedge D(x,y) \imp
                              ST_x(\vp_1)),
\end{array}
$$
\noindent where the arity of $P'$ is one greater than the arity of
$P$; the letter $P'$ is distinct from the letter $Q'$ if, and only if,
$P$ is distinct from $Q$; and all the newly introduced individual
variables are distinct from the previously used ones.  Intuitively,
the variable $x$ in $ST_x(\vp)$ stands for the world at which $\vp$ is
being evaluated.

Let
$$
\begin{array}{rcl}
M & = &
\exists x\,W(x)\wedge\forall x\,[W(x)\imp\exists y\,D(x,y)]
\wedge
\\
& &
{\wedge}\,\forall x\forall y\forall z\,
[W(x)\wedge W(y)\wedge \neg W(z) \wedge R(x,y) \wedge
D(x,z) \imp D(y,z)].
\end{array}
$$
The formula $M$
describes general properties of predicate Kripke frames; it says that
the set of worlds is non-empty, that the domain of every world is
non-empty, and that, if world $y$
is accessible from world $x$,
the domain of $x$ is included in the domain of $y$.

Let $F_n$ be a classical first-order formula in the signature
$\{R, =\}$ describing, for a fixed number $n \geqslant 2$, the
disjoint union $\frak{F}^{\uplus}_n$ of all the frames
$\frak{F}_m \in \frak{C}_0$ such that $m \leqslant n$.
Since $\frak{F}^{\uplus}_n$ is finite, the formula $F_n$ can be
effectively constructed,---all we need to do is say which worlds exist
in $\frak{F}^{\uplus}_n$, that those worlds are pairwise distinct,
that there are no other worlds in $\frak{F}^{\uplus}_n$, and describe
which worlds are, and which are not, related by the accessibility
relation.

Next, for an arbitrary classical first-order formula $\Phi$ in the
signature $\{R, =\}$, inductively define the formula $\Phi^\ast$, as
follows:

$$
\begin{array}{rcl}
\label{F*C}
  (x = y)^\ast & = & (x = y); \\
  {(R(x,y))}^{\ast} & = & R(x,y); \\
  {(\Phi_1 \con \Phi_2)}^\ast & = & \Phi_1^\ast \con \Phi_2^\ast; \\
  (\neg \Phi_1)^\ast & = & \neg \Phi_1^\ast; \\
  {(\forall x\,\Phi_1)}^\ast & = & \forall x\,(W(x)\imp\Phi_1^\ast).
\end{array}
$$

Lastly, given a predicate modal formula $\vp$, define

$$
\begin{array}{rcl}
  \widehat{\vp} & = & {M}\wedge F_{md(\vp) + 3}^\ast \imp \forall x\, (W(x) \imp ST_x(\varphi)).
\end{array}
$$

\begin{lemma}
\label{lem:Lw0intoQClE}
For every closed predicate modal formula $\vp$, the following holds:
$\vp \in L_0$ if, and only if, $\widehat{\vp} \in {\bf QCl}_=$.
\end{lemma}

\begin{proof}
  For the left-to-right, suppose that
  $\widehat{\vp} \notin {\bf QCl}_=$, i.e., $\widehat{\vp}$ fails in
  some classical first-order model $\frak{A}$.  We build a Kripke
  model, based on a $\frak{C}_0$-frame, refuting $\vp$.  Since
  $\frak{A} \Vdash M$ and $\frak{A} \Vdash F^{\ast}_{md(\vp) + 3}$, we
  can construct from $\frak{A}$ a predicate Kripke frame $\frak{F}_D$
  based on a frame isomorphic to $\frak{F}^{\uplus}_{md(\vp) + 3}$.
  Since $\frak{A} \not\Vdash \forall x\, (W(x) \imp ST_x(\varphi))$,
  for some assignment $g$, both $\frak{A} \Vdash^g W(x)$ and
  $\frak{A} \not\Vdash^g ST_x(\varphi)$ hold.  Recall that the
  standard translation has the property (see, e.g., \cite[Lemma
  3.12.2]{GShS}) that $\frak{A} \Vdash^g ST_x(\vp)$ if, and only if,
  \sat{M}{v}{\vp},
  where $\frak{A}$ and $\frak{M}$ are, respectively, classical and
  Kripke models agreeing on the interpretation of the predicate
  letters of $\vp$ and where $v$ is the world corresponding to the
  element $g(x)$ of $\frak{A}$.
  Therefore, we conclude that $\frak{M}, w \not\models \vp$, where
  $\frak{M}$ is a model based on $\frak{F}_{D}$ and $w$ is the world
  corresponding, under the isomorphism, to $g(x)$.  Since $\frak{M}$
  is based on a frame isomorphic to the frame
  $\frak{F}^{\uplus}_{md(\vp) + 3}$, which is a disjoint union of
  $\frak{C}_0$-frames, there exists a model $\frak{M}'$ based on a
  $\frak{C}_0$-frame
  such that $\frak{M}', w \not\models \vp$. Thus,
  $\vp \notin L_0$.


  For the converse, suppose that $\vp \not\in L_0$.  Then,
  $\frak{M}, \widehat{w} \not\models \vp$, for some model $\frak{M}$
  based on a $\frak{C}_0$-frame and some world $\widehat{w}$ in
  $\frak{M}$.  We show that, then, $\vp$ fails in a model based on a
  $\frak{C}_0$-frame $\frak{F}_m$, where $m \leqslant md(\vp) + 3$.


  If $\widehat{w} = w_1$, i.e., $\widehat{w}$ is the root of the frame
  underlying $\frak{M}$, we define the sought model as
  follows. Consider the chain $w_1 R w_2 R \ldots$ of worlds of
  $\frak{M}$; if it contains more than $s$ worlds, where
  $$
  s =
  \begin{cases}
    md(\vp) + 1, \mbox{~~if } md(\vp) \mbox{ is odd}, \\
    md(\vp) + 2, \mbox{~~if } md(\vp) \mbox{ is even},
  \end{cases}
  $$
  cut it off at $w_s$ to obtain the chain $w_1 R w_2 R \ldots R w_s$
  together with $w^\ast$ such that $w_1 R' w^\ast$.  Since
  $s \leqslant md(\vp) + 3$, we obtain that
  $\frak{M}', w_1 \not\models \vp$, for a model $\frak{M}'$ based on
  one of the disjunct frames from $\frak{F}^{\uplus}_{md(\vp) + 3}$.

  If, on the other hand, $\widehat{w}$ is not the root of the frame
  underlying $\frak{M}$, i.e., $\widehat{w} = w_k$, for some $k > 1$
  (notice that the case $\widehat{w} = w^\ast$ is identical to $k$
  being the maximal index of a world in $\frak{M}$), we define the
  sought model as follows. First, if the chain
  $w_k R w_{k+1} R \ldots$ contains more than $md(\vp) + 1$ worlds, we
  cut it off at $w_{k + md(\vp)}$.  Second, if $md(\vp)$ is even, we
  replace the chain $w_1 R \ldots R w_{k-1}$ together with $w^\ast$
  accessible from $w_1$ by worlds $w_0$ and $w^\ast$ with the
  accessibilities $w_0 R' w_k $ and $w_0 R w^\ast$; if, on the other
  hand, $md(\vp)$ is odd, we put in an additional world between $w_0$
  and $w_k$, so that the length of the resultant chain is even. In
  either case, the length of the chain of worlds in $\frak{M}'$ does
  not exceed $md(\vp) + 3$.  Hence, $\frak{M}', w_k \not\models \vp$,
  for a model $\frak{M}'$ based on one of the disjunct frames from
  $\frak{F}^{\uplus}_{md(\vp) + 3}$.

  Therefore, due to the aforementioned property of the standard
  translation, there exists a classical first-order model $\frak{A}$
  such that
  $\frak{A} \not\Vdash {M}\wedge F_{md(\vp) + 3}^\ast \imp \forall x\,
  (W(x) \imp ST_x(\varphi))$.
  Thus, we conclude that $\widehat{\vp} \notin {\bf QCl}_=$, as
  required.
\end{proof}

\begin{lemma}
  \label{lem:L0re}
  Logic $L_0$ is recursively enumerable.
\end{lemma}

\begin{proof}
  Immediate from Lemma~\ref{lem:Lw0intoQClE}.
\end{proof}

It remains to show that $L_0$ is not complete with respect to an
elementary class of rooted frames.

Given a positive integer $n$, let
$$
\alpha_n = \Diamond \Box {\bottom} \con \Diamond^n \Box {\bottom}.
$$
It is easy to check that the formula $\alpha_n$ is valid at worlds
that see a dead end in one step and also in $n$ steps.

\begin{lemma}
  \label{lem:alphas}
  Let $n$ be a positive integer. Then, $\neg \alpha_n \in L_0$ if, and
  only if, $n$ is even.
\end{lemma}

\begin{proof}
  Suppose $n$ is even.  Let $\frak{F}_m \in \frak{C}_0$ and let $w$ be
  a world in $\frak{F}_m$.  Assume that
  $\frak{F}_m, w \models \Diamond \Box {\bottom}$. Then, either
  $w = w_1$ or $w = w_{m-1}$ (notice that $m \geqslant 2$, since
  $\frak{F}_m \in \frak{C}_0$).  Neither $w_1$ nor $w_{m-1}$, however,
  can see a dead end in an even number of steps. Indeed, $w_{m-1}$ can
  only see a dead end in one step.  As for $w_1$, it can see two dead
  ends---one of them in one step, the other in $m-1$ steps; but $m-1$
  is odd, since $\frak{F}_m \in \frak{C}_0$ implies that $m$ is even.
  Therefore, $\frak{F}_m, w \models \neg \alpha_n$ holds for every $w$
  in $\frak{F}_m$.  Thus, $\neg \alpha_n \in L_0$.

  Suppose $n$ is odd. Then, $\frak{F}_{n+1}, w_1 \models \alpha_n$
  and, hence, $\frak{F}_{n+1}, w_1 \not\models \neg \alpha_n$. Since
  $\frak{F}_{n+1} \in \frak{C}_0$, we conclude that
  $\neg \alpha_n \notin L_0$.
\end{proof}

Now, let
$$
\begin{array}{rcccl}
  alt_2 & = & \Box p_1 \dis \Box (p_1 \imp p_2) \dis \Box (p_1 \con
              p_2 \imp p_3),
\end{array}
$$
where $p_1$, $p_2$, and $p_3$ are pairwise distinct propositional
variables, i.e., $0$-ary predicate letters.  The formula $alt_2$ is
valid at worlds that see at most two worlds (see, e.g.,
\cite[Proposition 3.45]{ChZ}).  Hence, $alt_2 \in L_0$.


\begin{lemma}
  \label{lem:L0notFOdef}
  Logic $L_0$ is not complete with respect to an elementary class of
  rooted frames.
\end{lemma}

\begin{proof}
  Assume otherwise, i.e, let $L_0$ be complete with respect to an
  elementary class $\frak{C}$ of rooted frames.

  First, we prove that, for every $n\geqslant 3$,
  $$
  \frak{F}_n\in \frak{C}~ \Longleftrightarrow ~ \mbox{$n$ is even.}
  $$
  Suppose $n$ is odd. Then, due to Lemma~\ref{lem:alphas},
  $\neg\alpha_{n-1}\in L_0$. Since
  $\frak{F}_n, w_1 \not\models\neg\alpha_{n-1}$, we obtain
  $\frak{F}_n\not\in \frak{C}$.

  Suppose $n$ is even. Then, due to Lemma~\ref{lem:alphas},
  $\neg\alpha_{n-1}\not\in L_0$. Thus, there exists a rooted frame
  $\frak{F}' \in\frak{C}$ and a world $w$ such that
  $\frak{F}',w \models\alpha_{n-1}$. We show that, up to isomorphism,
  $\frak{F}' = \frak{F}_n$.

  Let $\zeta = \Diamond p \imp \Box p$.  It is easy to check that
  $\zeta$ is valid at a world $w$ if, and only if, $w$ sees at most
  one world.  Since only the roots of $\frak{C}_0$-frames see more
  than one world,
  $\Box \zeta, \neg \zeta \imp \Diamond \Box {\bottom} \in L_0$.
  Since
  $alt_2, \Box \zeta, \neg \zeta \imp \Diamond \Box {\bottom} \in
  L_0$,
  if $\frak{F} \in \frak{C}$, then $\frak{F}$ has a branching degree
  of at most two, no world in $\frak{F}$ seen from another world sees
  more than one world (i.e., only the root of $\frak{F}$ may see two
  worlds), and if the root of $\frak{F}$ sees two worlds, one of them
  is a dead end.
  Since $\frak{F}', w \models\alpha_{n-1}$, the world $w$ is the root
  of $\frak{F}'$, and $w$ sees dead ends in one and $n-1$
  steps. Therefore, as claimed, $\frak{F}'$ is isomorphic to
  $\frak{F}_n$.

    Now, to obtain a contradiction, it remains to notice that
    classical first-order formulas cannot distinguish even and odd
    linear orders (see, e.g., \cite[Corollary 3.12]{Libkin}).
\end{proof}

We, thus, obtain the following:

\begin{theorem}
  There exists a normal predicate modal logic which is Kripke
  complete, recursively enumerable, and not complete with respect to
  an elementary class of rooted frames.
\end{theorem}

\begin{proof}
  Take $L_0$ as such a logic.
\end{proof}

We next show that the assumption of only considering rooted frames is
essential to the example presented above.

\begin{figure}
\centering
\begin{tikzpicture}[scale=1.5]

\coordinate (w1)   at (+0.0000,+0.0000);
\coordinate (w2)   at (+1.0000,+0.0000);
\coordinate (w3)   at (+2.0000,+0.0000);
\coordinate (w4)   at (+3.0000,+0.0000);
\coordinate (w5ph) at (+4.0000,+0.0000);
\coordinate (dts)  at (+4.5000,+0.0000);
\coordinate (w6ph) at (+5.0000,+0.0000);
\coordinate (w2k-1)at (+6.0000,+0.0000);
\coordinate (w2k)  at (+7.0000,+0.0000);
\coordinate (y)    at (+8.0000,+0.0000);
\coordinate (x)    at (+9.0000,+0.0000);
\coordinate (z)    at (+9.0000,-1.0000);
\coordinate (w*)   at (+0.0000,+1.0000);
\coordinate (w2*)  at (+1.0000,-1.0000);
\coordinate (w4*)  at (+3.0000,-1.0000);
\coordinate (w2k*) at (+7.0000,-1.0000);

\coordinate (p1)   at ( -0.3,-0.40);
\coordinate (p2)   at (1-0.3,-0.40);
\coordinate (p3)   at (1-0.3,-1.45);
\coordinate (p4)   at (1+0.3,-1.45);
\coordinate (p5)   at (1+0.3,+0.40);
\coordinate (p6)   at ( +0.3,+0.40);
\coordinate (p7)   at ( +0.3,+1.40);
\coordinate (p8)   at ( -0.3,+1.40);

\coordinate (p11)  at (2  -0.3,-0.40);
\coordinate (p21)  at (2+1-0.3,-0.40);
\coordinate (p31)  at (2+1-0.3,-1.45);
\coordinate (p41)  at (2+1+0.3,-1.45);
\coordinate (p51)  at (2+1+0.3,+0.40);
\coordinate (p61)  at (2  -0.3,+0.40);

\coordinate (p12)  at (6  -0.4,-0.40);
\coordinate (p22)  at (6+1-0.3,-0.40);
\coordinate (p32)  at (6+1-0.3,-1.45);
\coordinate (p42)  at (6+1+0.3,-1.45);
\coordinate (p52)  at (6+1+0.3,+0.40);
\coordinate (p62)  at (6  -0.4,+0.40);

\coordinate (p13)  at (8  -0.3,-0.40);
\coordinate (p23)  at (8+1-0.3,-0.40);
\coordinate (p33)  at (8+1-0.3,-1.45);
\coordinate (p43)  at (8+1+0.3,-1.45);
\coordinate (p53)  at (8+1+0.3,+0.40);
\coordinate (p63)  at (8  -0.3,+0.40);

\coordinate (text0) at (0.3,1.00);
\coordinate (text1) at (1.5,0.45);
\coordinate (text2) at (5.5,0.45);
\coordinate (text3) at (7.5,0.45);
\node [above right] at (text0)   {{``bottom part''} } ;
\node [above right] at (text1)   {{``regular part''}} ;
\node [above right] at (text2)   {{``regular part''}} ;
\node [above right] at (text3)   {{``regular part''}} ;

\begin{scope}[>=latex]
\draw [->,  shorten >= 1.5pt, shorten <= 1.5pt]
(w1) -- (w2);
\draw [->,  shorten >= 1.5pt, shorten <= 1.5pt]
(w2) -- (w3);
\draw [->,  shorten >= 1.5pt, shorten <= 1.5pt]
(w3) -- (w4);
\draw [->,  shorten >= 1.5pt, shorten <= 1.5pt]
(w4) -- (w5ph);
\draw [->,  shorten >= 1.5pt, shorten <= 1.5pt]
(w6ph) -- (w2k-1);
\draw [->,  shorten >= 1.5pt, shorten <= 1.5pt]
(w2k-1) -- (w2k);
\draw [->,  shorten >= 1.5pt, shorten <= 1.5pt]
(w2k) -- (y);
\draw [->,  shorten >= 1.5pt, shorten <= 1.5pt]
(y) -- (x);
\draw [->,  shorten >= 1.5pt, shorten <= 1.5pt]
(z) -- (x);
\draw [->,  shorten >= 1.5pt, shorten <= 1.5pt]
(w2*) -- (w2);
\draw [->,  shorten >= 1.5pt, shorten <= 1.5pt]
(w4*) -- (w4);
\draw [->,  shorten >= 1.5pt, shorten <= 1.5pt]
(w2k*) -- (w2k);
\draw [->,  shorten >= 1.5pt, shorten <= 1.5pt]
(w1) -- (w*);
\end{scope}

\draw[dashed, rounded corners=1.5] 
(p1) -- (p2) -- (p3) -- (p4) -- (p5) -- (p6) -- (p7) -- (p8) -- cycle;
\draw[dashed, rounded corners=1.5]
(p11) -- (p21) -- (p31) -- (p41) -- (p51) -- (p61) -- cycle;
\draw[dashed, rounded corners=1.5]
(p12) -- (p22) -- (p32) -- (p42) -- (p52) -- (p62) -- cycle;
\draw[dashed, rounded corners=1.5]
(p13) -- (p23) -- (p33) -- (p43) -- (p53) -- (p63) -- cycle;

\node [below      ] at (w1)   {${w_1}$}         ;
\node [above      ] at (w2)   {${w_2}$}         ;
\node [above      ] at (w3)   {${w_3}$}         ;
\node [above      ] at (w4)   {${w_4}$}         ;
\node [above      ] at (w*)   {${w^\ast}$}      ;
\node [           ] at (dts)  {${\cdots}$}      ;
\node [above      ] at (w2k-1){${w_{2k-1}}$}    ;
\node [above      ] at (w2k)  {${w_{2k}}$}      ;
\node [above      ] at (y)    {${y}$}           ;
\node [above      ] at (x)    {${x}$}           ;
\node [below      ] at (z)    {${z}$}           ;
\node [below      ] at (w2*)  {${w_2^\ast}$}    ;
\node [below      ] at (w4*)  {${w_4^\ast}$}    ;
\node [below      ] at (w2k*) {${w_{2k}^\ast}$} ;

\filldraw [] (w1)   circle [radius=1.5pt];
\filldraw [] (w2)   circle [radius=1.5pt];
\filldraw [] (w3)   circle [radius=1.5pt];
\filldraw [] (w4)   circle [radius=1.5pt];
\filldraw [] (w*)   circle [radius=1.5pt];
\filldraw [] (w2k-1)circle [radius=1.5pt];
\filldraw [] (w2k)  circle [radius=1.5pt];
\filldraw [] (x)    circle [radius=1.5pt];
\filldraw [] (y)    circle [radius=1.5pt];
\filldraw [] (z)    circle [radius=1.5pt];
\filldraw [] (w2*)  circle [radius=1.5pt];
\filldraw [] (w4*)  circle [radius=1.5pt];
\filldraw [] (w2k*) circle [radius=1.5pt];

\end{tikzpicture}

\caption{Regular construction for $\frak{C}_0^\ast$-frames}
\label{C0ast}
\end{figure}
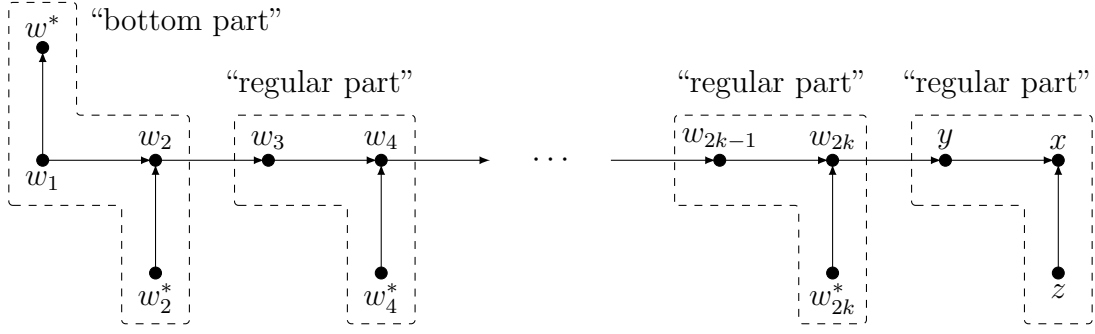

\begin{proposition}
  There exists an elementary class $\frak{C}^\ast_0$ of frames which
  are not rooted such that $L_0 = L(\frak{C}^\ast_0)$.
\end{proposition}

\begin{proof}
  We construct $\frak{C}^\ast_0$ so that it contains frames that
  resemble those in $\frak{C}_0$, but where every world $w_{2k}$ is
  marked off by an additional world $w^\ast_{2k}$, which sees
  $w_{2k}$.  This enables us to tell apart evenly and oddly numbered
  elements of the chain of worlds of a frame.  Then, we make sure that
  a chain of a frame in $\frak{C}^\ast_0$ does not end with an oddly
  numbered world.  Notice that $\frak{C}^\ast_0$ might contain frames
  that do not look exactly like $\frak{C}_0$-frames with
  ``additional'' worlds $w^\ast_{2k}$; in particular, it might contain
  infinite frames; this is, however, immaterial to our argument.

  We now describe $\frak{C}^\ast_0$-frames with classical first-order
  formulas (also, see Figure~\ref{C0ast}).

  First, we say that $\frak{C}^\ast_0$-frames are irreflexive and do
  not contain transitive chains with more than two elements:
  $$
  \Phi_1 = \forall x\,\neg R(x,x) \wedge \forall x\forall y\forall
  z\,(R(x,y)\wedge R(y,z) \imp \neg R(x,z)).
  $$
  Second, we describe the ``bottom part'' of a
  $\frak{C}^\ast_0$-frame, which looks as follows: the bottom-most
  world $w_1$ sees two worlds, $w^\ast$ and $w_2$, which is also seen
  from $w_2^\ast$, all these worlds being distinct (that $w_2$ is only
  seen from $w_1$ and $w_2^\ast$ follows from the formula $\Phi_3$
  below):
  $$
  \begin{array}{rcl}
    \Phi_2 (w_1) & = &
                       \exists w^\ast \exists w_2\exists w^\ast_2\,
                       [
                       \forall x\,(R(w_1,x)\leftrightarrow x=w^\ast\vee x=w_2)
    \\&&
         {\con}~R(w^\ast_2,w_2)
         \con \neg \exists x\, R(x, w_1)
         \con \neg \exists x\, R(x, w^\ast_2)
    \\&&
         {\con}~\forall x\,(R(x,w^\ast)\imp x = w_1)
         \con \neg \exists x\, R(w^\ast,x)
         \con w_1 \ne w^\ast_2
         ].
  \end{array}
  $$
  Third, we say that the bottom-most world $w_1$ is the only one that
  can see more than one world and that no world is seen from more than
  two worlds:
  $$
  \begin{array}{rcl}
    \Phi_3 (w_1) & = &
                       \forall x\,(x\ne w_1 \imp \neg\exists y\exists z\,(y\ne z\wedge R(x,y)\con R(x,z))) ~{\con}
    \\&&
         \forall w\neg\exists x\exists y\exists z\,(x\ne y \con y\ne z\con x\ne z \con R(x,w)\con R(y,w)\con R(z,w)).
  \end{array}
  $$
  Lastly, we describe the repetitive procedure of extending an
  $\frak{C}^\ast_0$-frame by appending to the topmost world $w$ two
  worlds $y$ and $x$, distinct from $w$ and each other, such that
  $wRyRx$ holds, as well as a world $z$ that lies off the main chain
  and sees $x$ (intuitively, $z$ marks off $x$ as the topmost world):
  $$
  \begin{array}{rcl}
    \Phi_4 & = &
                 \forall w\,[\exists y\exists z\, (y\ne z \con R(y,w)\con R(z,w)) \imp
    \\&&
         \neg\exists x\,R(w,x) ~{\vee}
    \\&&
         \exists x\exists y\exists z\,(R(w,y)\con R(y,x)\con R(z,x) \con
         y\ne z\con x\ne w ~{\con}
    \\&&
         \forall u\,(R(u,y)\imp u = w) \con \neg\exists u\,R(u,z))].
  \end{array}
  $$
  The procedure of extending a $\frak{C}^\ast_0$-frame described by
  $\Phi_4$ can be carried out arbitrarily, including infinitely, many
  times.

  Then, the class $\frak{C}^\ast_0$ is defined to contain the frames
  satisfying the following classical first-order formula:
  $$
  \exists w_1 ( \Phi_1 \con \Phi_2 (w_1) \con \Phi_3 (w_1) \con \Phi_4).
  $$
  
  We claim that $L(\frak{C}^\ast_0) = L(\frak{C}_0)$.

  To see that $L(\frak{C}_0) \subseteq L(\frak{C}^\ast_0)$, let
  $\frak{F}, w \not\models \vp$, for some $\frak{F}\in\frak{C}^\ast_0$
  and some $w$ in $\frak{F}$.  Let $\frak{F}_w$ be a subframe of
  $\frak{F}$ generated by $w$.

  If $w \ne w_1$, since the branching factor of $\frak{F}$ is two and
  only the bottom-most world $w_1$ of $\frak{F}$ can see two worlds,
  $\frak{F}_w$ is a chain of worlds.  If we take the initial segment
  $\frak{F}_w'$ of this chain of length at most $md(\vp) + 1$, then
  clearly, $\frak{F}_w'$ is isomorphic to a generated subframe of some
  frame in $\frak{C}_0$ and, moreover,
  $\frak{F}_w', w \not\models \vp$.  Hence, in this case,
  $\vp \notin L(\frak{C}_0)$.

  If, on the other hand, $w= w_1$, the world $w$ sees a dead end
  $w^\ast$ and also sees a chain $w_1Rw_2R\ldots$ of worlds that is
  built in repetitive stages, starting from $w_1$ and $w_2$ and being
  extended at every stage by exactly two worlds; hence, the chain
  $w_1Rw_2R\ldots$ is either infinite or contains an even number of
  worlds. In either case, we can find a frame
  $\frak{F} \in \frak{C}_0$ such that $\frak{F}, w_1 \not\models \vp$;
  thus, if $w= w_1$, we also obtain $\vp \notin L(\frak{C}_0)$.

  To see that $L(\frak{C}^\ast_0) \subseteq L(\frak{C}_0)$, notice
  that, given $\frak{F}_n \in\frak{C}_0$, where
  $\frak{F}_n = \langle W_n, R_n \rangle$, if we make every world
  $w_{2k} \in {W}_n$ accessible from an additional world
  $w^\ast_{2k}$, then the resulting frame $\frak{F}^\ast_n$ is
  in~$\frak{C}^\ast_0$.  Since, for every $w \in W_n$, the subframe of
  $\frak{F}_n$ generated by $w$ coincides with the subframe of
  $\frak{F}^\ast_n$ generated by $w$, we obtain that
  $\vp \notin L(\frak{C}_0)$ implies $\vp \notin L(\frak{C}^\ast_0)$,
  as required.

  Since $L_0 = L(\frak{C}_0)$, we conclude that
  $L_0 = L(\frak{C}^\ast_0)$.
\end{proof}

Thus, the assumption of all frames under consideration being rooted is
not insignificant in the context of the present enquiry. We do,
however, show in the next section that, if we discard the restriction
to rooted frames, the answer to the main question considered in this
paper remains negative.

\section{An example over arbitrary frames}
\label{sec:example-2}

In this section, we exhibit the normal predicate modal logic $L_1$
that is recursively enumerable, Kripke complete, but not complete with
respect to an elementary class of any frames.

Given $n \geqslant 2$, let $\frak{G}_n$ be the frame
$\langle W_n, R_n \rangle$, where
$W_n = \{ w_1, w_2, \ldots, w_n, w^\ast \}$ and
$R_n = \{ \langle w_i, w_{i+1} \rangle : 1 \leqslant i < n \} \cup \{
\langle w_n, w_1 \rangle, \langle w_1, w^\ast \rangle \}$.
In other words, $\frak{G}_n$ is a ring made up of $n$ worlds, such
that one world of the ring sees a dead end; the frame $\frak{G}_n$ is
depicted in Figure~\ref{fig-2}. Denote the set of all such frames by
$\frak{C}'$, and let
$\frak{C}_1 = \{ \frak{G}_n \in \frak{C}' : n \mbox{ is even}\}$;
finally, let $L_1 = L(\frak{C}_1)$.

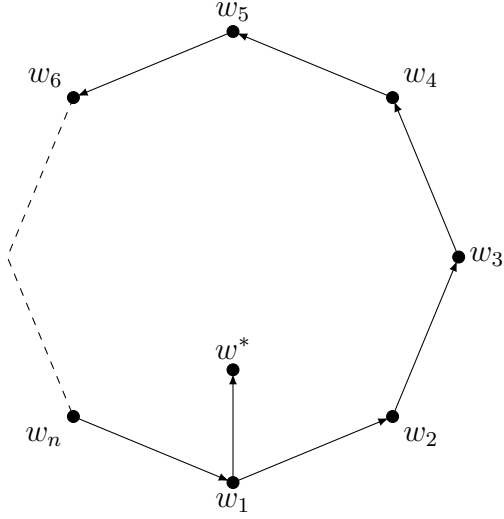
\begin{figure}
\centering
\begin{tikzpicture}[scale=1.5]

\coordinate (w1) at (+0.0000,-2.0000);
\coordinate (w2) at (+1.4142,-1.4142);
\coordinate (w3) at (+2.0000,+0.0000);
\coordinate (w4) at (+1.4142,+1.4142);
\coordinate (w5) at (+0.0000,+2.0000);
\coordinate (w6) at (-1.4142,+1.4142);
\coordinate (w*) at (+0.0000,-1.0000);
\coordinate (wn) at (-1.4142,-1.4142);
\coordinate (w7) at (-2.0000,+0.0000);

\begin{scope}[>=latex]
\draw [->,  shorten >= 1.5pt, shorten <= 1.5pt]
(wn) -- (w1);
\draw [->,  shorten >= 1.5pt, shorten <= 1.5pt]
(w1) -- (w2);
\draw [->,  shorten >= 1.5pt, shorten <= 1.5pt]
(w2) -- (w3);
\draw [->,  shorten >= 1.5pt, shorten <= 1.5pt]
(w3) -- (w4);
\draw [->,  shorten >= 1.5pt, shorten <= 1.5pt]
(w4) -- (w5);
\draw [->,  shorten >= 1.5pt, shorten <= 1.5pt]
(w5) -- (w6);
\draw [->,  shorten >= 1.5pt, shorten <= 1.5pt]
(w1) -- (w*);
\end{scope}

\draw[dashed, rounded corners=1.5] (w6) -- (w7) -- (wn);

\node [below      ] at (w1) {${w_1}$}           ;
\node [below right] at (w2) {${w_2}$}           ;
\node [      right] at (w3) {${w_3}$}           ;
\node [above right] at (w4) {${w_4}$}           ;
\node [above      ] at (w5) {${w_5}$}           ;
\node [above left ] at (w6) {${w_6}$}           ;
\node [above      ] at (w*) {${w^\ast}$}        ;
\node [below left ] at (wn) {${w_n}$}           ;
\node [      left ] at (w7) {${\phantom{w_7}}$} ;


\filldraw [] (w1) circle [radius=1.5pt];
\filldraw [] (w2) circle [radius=1.5pt];
\filldraw [] (w3) circle [radius=1.5pt];
\filldraw [] (w4) circle [radius=1.5pt];
\filldraw [] (w5) circle [radius=1.5pt];
\filldraw [] (w6) circle [radius=1.5pt];
\filldraw [] (w*) circle [radius=1.5pt];
\filldraw [] (wn) circle [radius=1.5pt];

\end{tikzpicture}

\caption{Frame $\frak{G}_n$}
\label{fig-2}
\end{figure}

To show that $L_1$ is recursively enumerable, we effectively embed it
into the classical first-order logic with equality ${\bf QCl}_=$.  The
embedding is very similar to the one defined in the preceding section
for $L_0$.  The only difference is in the description, using a
classical first-order formula, of a particular frame that fails a
modal formula in case the corresponding classical formula is not
valid.

Let $G_n$ be a classical first-order formula in the signature
$\{R, =\}$ that, for a fixed number $n \geqslant 2$, describes the
disjoint union $\frak{G}^{\uplus}_n$ of all the frames $\frak{G}_m$ in
$\frak{C}_1$ such that $m \leqslant n$.  Since $\frak{G}^{\uplus}_n$
is finite, $G_n$ can be effectively constructed.

Let the formula $M$, as well as the translations $\cdot^\ast$ and
$ST_x(\cdot)$, be defined as previously. Given a predicate modal
formula $\vp$, define

$$
\begin{array}{rcl}
  \bar{\vp} & = & {M}\wedge G_{md(\vp) + 3}^\ast \imp \forall x\, (W(x) \imp ST_x(\varphi)).
\end{array}
$$

\begin{lemma}
\label{lem:L1intoQClE}
For every closed predicate modal formula $\vp$, the following holds:
$\vp \in L_1$ if, and only if, $\bar{\vp} \in {\bf QCl}_=$.
\end{lemma}

\begin{proof}
  The left-to-right implication is argued as in
  Lemma~\ref{lem:Lw0intoQClE}.

  For the right-to-left, assume that $\vp \notin L_1$, i.e.,
  $\frak{M}, \bar{w} \not\models \vp$, for some model $\frak{M}$
  based on a $\frak{C}_1$-frame, say $\frak{G}_n$, and some world
  $\bar{w}$ in $\frak{M}$.  We show that, then, $\vp$ fails in a
  model based on a $\frak{C}_1$-frame $\frak{G}_m$, where
  $m \leqslant md(\vp) + 3$.

  If $\bar{w} = w^\ast$, then clearly
  $\mmodel{M}', w^\ast \not\models \vp$ holds for a model
  $\mmodel{M}'$ based on the frame $\frak{G}_2$.

  Assume, on the other hand, that $\bar{w} = w_k$, for some
  $k \in \{1, \ldots, n\}$.  If $n \leqslant md(\vp) + 3$, then we are
  done.  Otherwise, let $\frak{G}'$ be either $\frak{G}_{md(\vp) + 2}$
  or $\frak{G}_{md(\vp) + 3}$, whichever belongs to $\frak{C}_1$ (one
  of them, clearly, does).  It is easy to see that $\frak{G}'$
  contains a world $w'$ such that the subframe of $\frak{G}'$ made up
  of worlds reachable from $w'$ in at most $md(\vp)$ steps is
  isomorphic to the subframe of $\frak{G}_n$ made up of worlds
  reachable from $\bar{w}$ in at most $md(\vp)$ steps.  Hence,
  $\frak{M}', w' \not\models \vp$ holds for a model $\frak{M}'$ based
  on $\frak{G}'$.

  Therefore, due to the property of the standard translation mentioned
  in the proof of Lemma~\ref{lem:Lw0intoQClE}, there exists a
  classical first-order model $\frak{A}$ such that
  $\frak{A} \not\Vdash {M}\wedge G_{md(\vp) + 3}^\ast \imp \forall x\,
  (W(x) \imp ST_x(\varphi))$.
  Thus, $\widehat{\vp} \notin {\bf QCl}_=$.
\end{proof}

\begin{lemma}
  \label{lem:L1re}
  Logic $L_1$ is recursively enumerable.
\end{lemma}

\begin{proof}
  Immediate from Lemma~\ref{lem:L1intoQClE}.
\end{proof}

It remains to show that $L_1$ is not complete with respect to an
elementary class of frames.  To that end, we define the following
formulas:
$$
\begin{array}{lclr}
  \beta_n & = & \Diamond \Box {\bottom} \con \Diamond^n \Diamond \Box
                {\bottom} \con \bigwedge\limits_{k=1}^{n-1} \neg
                \Diamond^k  \Diamond \Box
                {\bottom};
  \\
  \gamma & = & \Diamond \Box {\bottom} \dis (\Diamond p \imp \Box p);
         & \\
  \delta^k_{n} & = & \Diamond^k \beta_n \con p \imp \Diamond^n p;
         & \\
  \varepsilon_n & = & \beta_n \con p \imp \Box^n (\beta_n \con p),
\end{array}
$$
where $p$ is a propositional variable.

\begin{lemma}
  \label{lem:formulas}
  The formula $\beta_n$ is valid at a world $w$ if, and only if, $w$
  can see a dead end, can see in $n$ steps a world that sees a dead
  end, and cannot see in any number $k \in \{1, \ldots, n-1\}$ of
  steps a world that sees a dead end.

  The formula $\gamma$ is valid on a frame $\frak{F}$ if, and only if,
  only those worlds of $\frak{F}$ that see a dead end can see more
  than one world.

  The formula $\delta^k_{n}$ is valid on a frame $\frak{F}$ if, and
  only if, every world in $\frak{F}$ that can see in $k$ steps a world
  $w$ such that $w \models \beta_n$ can see itself in $n$ steps.

  The formula $\varepsilon_n$ is valid on a frame $\frak{F}$ if, and
  only if, every world $w$ of $\frak{F}$ such that $w \models \beta_n$
  cannot see in $n$ steps a world $w' \ne w$.
\end{lemma}

\begin{proof}
  We only remark on $\varepsilon_n$, leaving the rest to the reader.

  Notice that the truth status of $\beta_n$ at a world does not depend
  on the interpretation.  If $\frak{F}$ contains a world $w$ such that
  $w \models \beta_n$ and $w$ can see in $n$ steps a word $w' \ne w $,
  then $\frak{M}, w \not\models \varepsilon_n$ if $\frak{M}$ is a
  model based on $\frak{F}$ such that $\frak{M}, v \models p$ if, and
  only if, $v = w$.
\end{proof}

Recall that the formula $alt_2$ is valid at worlds that see at most
two worlds.

\begin{lemma}
  \label{lem:valid_in_L1}
  The following formulas belong to $L_1$: $alt_2$; $\gamma$;
  $\delta^k_{n}$, for every $k, n$; $\varepsilon_n$, for every $n$.
\end{lemma}

\begin{proof}
  It is straightforward to check that every $\frak{C}_1$-frame
  satisfies the properties associated with the formulas listed in the
  statement of the lemma.
\end{proof}

\begin{lemma}
  \label{lem:betas}
  Let $n$ be a natural number such that $n \geqslant 1$. Then,
  $\neg \beta_n \in L_1$ if, and only if, $n$ is odd.
\end{lemma}

\begin{proof}
  Suppose $n$ is odd.  Let $\frak{G}_m \in \frak{C}_1$ and let $w$ be
  a world in $\frak{G}_m$.  Assume that
  $\frak{G}_m, w \models \Diamond \Box {\bottom}$. Due to
  Lemma~\ref{lem:formulas}, this is only possible if $w = w_1$.  Since
  in no $\frak{C}_1$-frame can $w_1$ see a dead end in an odd number
  of steps, $\frak{G}_m, w \models \neg \beta_n$ holds for every $w$
  in $\frak{G}_m$; thus, $\neg \beta_n \in L_1$.

  Suppose $n$ is even. Then, due to Lemma~\ref{lem:formulas},
  $\frak{G}_{n}, w_1 \models \beta_n$ and, hence,
  $\frak{G}_{n}, w_1 \not\models \neg \beta_n$. Since
  $\frak{G}_{n} \in \frak{C}_1$, we conclude that
  $\neg \beta_n \notin L_1$.
\end{proof}

\begin{lemma}
  \label{lem:L1notFOdef}
  Logic $L_1$ is not complete with respect to an elementary class of
  frames.
\end{lemma}

\begin{proof}
  Assume otherwise, i.e, let $L_1$ be complete with respect to an
  elementary class $\frak{C}$ of frames.  Further assume that
  $\frak{C}$ is defined by a classical first-order formula $\Phi$ with
  quantifier rank $n$.

  Since ${2^n}$ is even, due to Lemma~\ref{lem:betas},
  $\neg \beta_{2^n} \notin L_1$.  Therefore, $\frak{C}$ contains a
  frame $\frak{F} = \langle W, R \rangle$ and $w_1 \in W$ such that
  $\frak{F}, w_1 \models \beta_{2^n}$.  Thus, $\frak{F}$ contains
  worlds $w_1, \ldots, w_{2^n + 1}, w_1^\ast, w_{2^n + 1}^\ast$ such
  that $w_1 R w_2 \ldots w_{2^n} R w_{2^n + 1}$, as well as
  $w_1 R w_1^\ast$ and $w_{2^n + 1} R w_{2^n + 1}^\ast$. Since
  $\frak{F} \models \gamma$, every $w_i \in \{ w_2, \ldots, w_{2^n}\}$
  can see only one world, which is, thus, $w_{i+1}$. As
  $\frak{F} \models \varepsilon_{2^n}$, due to
  Lemma~\ref{lem:formulas}, $w_1 = w_{2^n+1}$.  Since
  $\frak{F} \models alt_2$, we obtain $w^\ast_{2^n+1} = w^\ast_1$ or
  $w^\ast_{2^n+1} = w_2$.  Since $\frak{F} \models \delta^k_{2^n}$,
  for every $k$, no world in $X = \{w_1, \ldots w_{2^n} \}$ is seen
  from a world in $W \setminus X$.  Indeed, suppose otherwise, i.e.,
  let $v R w$, for some $v \in W \setminus X$ and some $w \in X$.
  Then, $v R^k w_1$ holds for some $k \in \mathds{N}$. Since
  $\frak{F}, v \models \delta^k_{2^n}$, there is a path of length
  $2^n$ from $v$ to $v$. Notice that the said path cannot contain
  worlds from $X$; hence, $w$ is not part of this path, and thus $v$
  can see at least two worlds neither of which is a dead end.  But, as
  $\frak{F}, v \models \gamma$, the world $v$ also sees a dead end,
  which contradicts the fact that $\frak{F}, v \models alt_2$.

  Thus, $\frak{F}$ looks like a ring of $2^n$ worlds, where one world
  in the ring also sees a dead end, and where no world not in the ring
  can see a world in the ring.  Now, consider a frame $\frak{F}'$ that
  looks like $\frak{F}$, except that its ring is made up of $2^{n}+1$,
  rather than $2^{n}$, worlds.  Since, due to Lemma~\ref{lem:betas},
  $\neg \beta_{2^{n}+1} \in L_1$, we obtain
  $\frak{F}' \notin \frak{C}$.  One can, however, use
  Ehrenfeucht--Fra\"{i}ss\'{e} games (see, e.g., \cite[Chapter
  3]{Libkin}) to show that $\Phi$ cannot distinguish $\frak{F}$ from
  $\frak{F}'$, which gives us a contradiction.
\end{proof}

\begin{theorem}
  There exists a normal predicate modal logic which is Kripke
  complete, recursively enumerable, and not complete with respect to
  an elementary class of frames.
\end{theorem}

\begin{proof}
  Take $L_1$ as such a logic.
\end{proof}

\section{Conclusion}
\label{sec:discussion}

The main result of the present paper is the construction of an example
of a normal predicate modal logic that is Kripke complete, recursively
enumerable, and not complete with respect to an elementary class of
Kripke frames, which solves the problem left open in~\cite{RShAiML}.
Notice that our example is a conservative extension of the classical
first-order logic, as no restrictions on either the sizes of the
domains of the worlds or the possible interpretations of predicate
letters have been used in its construction.  Also notice that a
similar example can be constructed, along the same lines, of a logic
of predicate frames with constant domains.



It remains unclear to us whether a similar example can be constructed
in the extensions of the intuitionistic predicate logic.  A step
toward answering that question would the construction of a normal
predicate modal logic, satisfying the properties studied in this
paper, whose frames are reflexive and transitive, as is the
accessibility relation in the extensions of the intuitionistic logic.
While it is not difficult to modify the example presented in
Section~\ref{sec:example-2} to construct a logic, satisfying the same
properties, of reflexive frames, it is not clear to us whether an
example over reflexive and transitive---or, simply,
transitive---frames exists.


\section*{Acknowledgements}

We would like to thank Philippe Balbiani, Andreas Herzig, and the
participants of the research seminar at Institut de Recherche en
Informatique de Toulouse, Universit\'e Paul Sabatier, for a discussion
of the ideas that led to the results presented in this paper.  We
would also like to thank Valentin Shehtman and the participants of the
research seminar at the Department of Mathematical Logic, Lomonosov
Moscow State University, for a discussion of the ideas presented here.
We thank the anonymous referees for the comments and suggestions that
helped to improve the paper.



\selectlanguage{english}

\begin{thebibliography}{1}

\bibitem{ChZ}
Alexander Chagrov and Michael Zakharyaschev.
\newblock {\em Modal Logic}.
\newblock Oxford University Press, 1997.

\bibitem{GShS}
Dov Gabbay, Valentin Shehtman, and Dmitrij Skvortsov.
\newblock {\em Quantification in Nonclassical Logic, Volume 1}.
\newblock Elsevier, 2009.

\bibitem{Libkin}
Leonid Libkin.
\newblock {\em Elements of Finite Model Theory}.
\newblock Springer, 2004.

\bibitem{Montagna84}
Franco Montagna.
\newblock The predicate modal logic of provability.
\newblock {\em Notre Dame Journal of Formal Logic}, 25(2):179--189, 1984.

\bibitem{RShAiML}
Mikhail Rybakov and Dmitry Shkatov.
\newblock A recursively enumerable {K}ripke complete first-order logic not
  complete with respect to a~first-order definable class of frames.
\newblock In Guram Bezhanishvili, Giovanna D'Agostino, George Metcalfe, and
  Thomas Studer, editors, {\em Advances in Modal Logic}, volume~12, pages
  531--540. College Publications, 2018.

\bibitem{Rybakov01} \selectlanguage{russian} Михаил~Рыбаков.
  \newblock \selectlanguage{russian} Перечислимость модальных
  предикатных логик и условия обрыва возрастающих цепей.  \newblock
  {\em \selectlanguage{russian} Логические
    исследования\selectlanguage{english}}, {\selectlanguage{russian}
    выпуск 8\selectlanguage{english}}, 155--167, 2001.
  \newblock (Mikhail Rybakov. Enumerability of predicate modal logics
  and conditions of non-existence of infinite ascending chains. {\em
    Logicheskiye Issledovaniya}, 8:155--167, 2001. In Russian.)

\bibitem{ChR00} \selectlanguage{russian}Михаил Рыбаков{,}
  Александр~Чагров.  \newblock \selectlanguage{russian}Стандартные
  переводы неклассических формул и относительная разрешимость логик.
  \newblock {\em \selectlanguage{russian}Труды
    научно-исследовательского семинара Логического центра Института
    философии Российской академии наук}, {\selectlanguage{russian}
    выпуск XIV \selectlanguage{english}}, 81--98, 2000.  \newblock
  (Mikhail Rybakov and Alexander Chagrov.  \newblock Standard
  translations of non-classical formulas and relative decidability of
  logics.  \newblock {\em Annals of the research seminar of the Center
    for Logic of the Institute of Philosophy of the Russian Academy of
    Sciences}, {X}{I}{V}, 81--98, 2000. In Russian.)

\bibitem{Skv88}
Dmitrij Skvortsov.
\newblock On axiomatizability of some intermediate predicate logics.
\newblock {\em Reports on Mathematical Logic}, 22:115--116, 1988.
  
\bibitem{Skvortsov95}
Dmitrij Skvortsov.
\newblock On the predicate logics of finite {K}ripke frames.
\newblock {\em Studia Logica}, 54(1):79--88, 1995.

\bibitem{WZ01}
Frank Wolter and Michael Zakharyaschev.
\newblock Decidable fragments of first-order modal logics.
\newblock {\em The Journal of Symbolic Logic}, 66(3):1415--1438, 2001.

\end{thebibliography}

\end{document}